\documentclass[11pt]{amsart}
\usepackage{mathrsfs}
\usepackage{amsfonts}
\usepackage{latexsym,amsmath,amssymb}

 \renewcommand{\epsilon}{\varepsilon}

\newtheorem{theorem}{Theorem}[section]

 \newtheorem{Example}[theorem]{Example}
 \newtheorem{lemma}[theorem]{Lemma}

 \newtheorem{proposition}[theorem]{proposition}
 
\newtheorem{deff}[theorem]{Definition}
 \newtheorem{rem}[theorem]{Remark}
 \newcommand{\bth}{\begin{theorem}}
 \newcommand{\ble}{\begin{lemma}}
 \newcommand{\bcor}{\begin{corr}}
 \newcommand{\bdeff}{\begin{deff}}
 \newcommand{\bprop}{\begin{proposition}}
 \newcommand{\ele}{\end{lemma}}
 \newcommand{\ecor}{\end{corr}}
 \newcommand{\edeff}{\end{deff}}
 
 \newcommand{\eprop}{\end{proposition}}

 \renewcommand{\Pi}{\varPi}

 \renewcommand{\epsilon}{\varepsilon}

\numberwithin{equation}{section}

\thanks{The first author is supported  by National Science Foundation of China(No.11771103) and  Guangxi Natural Science Foundation(No.2017GXNSFFA198017).
The second author is supported by Guangdong Natural Science Foundation(No.2016A030307008).}

\title
[Nonlinear second boundary conditions ]{A convergence result on the second boundary value problem for parabolic equations}
\author{Rongli Huang}
\address{School of Mathematics and Statistics, Guangxi Normal University,
Guilin, Guangxi 541004, People's Republic of China,
 E-mail: ronglihuangmath@gxnu.edu.cn}
 
 \author{Yunhua Ye}
\address{School of Mathematics, Jiaying University,
Meizhou, Guangdong 514015, People's Republic of China, E-mail: mathyhye@163.com}
\date{}

\begin{document}
\maketitle
\begin{abstract}
We establish a Schn$\ddot{\text{u}}$rer's convergence result and then apply it to obtain the existence of  solutions on the second boundary value problem for a family of special Lagrangian equations.
\end{abstract}

\let\thefootnote\relax
\section{Introduction}
Let $\Omega$ be a bounded domain with smooth boundary in $\mathbb{R}^{n}$ and $\mathcal{S}$ be the open connected subset of $\mathbb{S}_{n}$ where
$$  \mathbb{S}_{n}=\{n\times n\,\, \text{real\,\,symmetric\,\,matrix}\}.$$
Given $u_{0}: \Omega\rightarrow\mathbb{R}$,
we consider nonlinear parabolic equation with second boundary condition:
\begin{equation}\label{e1.1}
\left\{ \begin{aligned}u_{t}-F(D^{2}u,Du,x)&=0,
&  x\in \Omega,\,\, t>0, \\
h(Du,x)&=0,& x\in \partial\Omega,\,\, t>0,\\
u&=u_{0},& x\in \Omega,\,\, t=0,
\end{aligned} \right.
\end{equation}
where $F: \mathcal{S}\times \mathbb{R}^{n}\times\Omega\rightarrow \mathbb{R}$ is $C^{2+\alpha}$ for some $0<\alpha<1$ and
satisfies
\begin{equation}\label{e1.2}
A<B\Rightarrow F(A,p,x)<F(B,p,x).
\end{equation}
In other words, (\ref{e1.2}) means that $F$ is non-decreasing with respect to the matrix variable and we say that
$F$ is strictly parabolic. Here $h:\mathbb{R}^{n}\times\partial\Omega\rightarrow \mathbb{R}$ is $C^{1}$ and satisfies
\begin{equation}\label{e1.3}
\sum_{k=1}^{n}h_{p_{k}}(\cdot,x)\nu_{k}>0,
\end{equation}
 where $\nu=(\nu_{1},\nu_{2}, \cdots,\nu_{n})$ is the unit inward normal vector of $\partial\Omega$.
 (\ref{e1.3}) shows that $h(Du,x)=0$ is  oblique.

Throughout the following,  Einstein's convention of
 summation over repeated indices will be adopted.
Denote $$u_{i}=\dfrac{\partial u}{\partial x_{i}},
u_{ij}=\dfrac{\partial^{2}u}{\partial x_{i}\partial x_{j}},
u_{ijk}=\dfrac{\partial^{3}u}{\partial x_{i}\partial x_{j}\partial
x_{k}}, \cdots $$ and
$$[u^{ij}]=[ u_{ij}]^{-1},\,\,\,
F^{ij}(D^{2}u)=\frac{\partial F}{\partial u_{ij}},\,\,\,h_{p_{k}}(Du)=\frac{\partial h}{\partial u_{k}}, \,\,\, \Omega_{T}=\Omega\times(0,T). $$

In this paper, we are concerned with the convergence of nonlinear parabolic equation \eqref{e1.1} to translating solution under some a-priori estimates as $t\rightarrow \infty$.
The translating solutions are intimately related to the solutions of minimal Lagrangian diffeomorphism problem. In general,
evolution equations often have special solutions called solitons which keep their shape during the evolution. For examples, two very important classes of solitons
in mean curvature flow are self-shrinker and translating solutions which evolve by a homothety or a translation respectively.  Translating solutions are interesting examples
of the evolution equations since they are precise solutions in the sense that their evolution is known, which is very hard to determine
in general.
Our main result concerning the asymptotic behavior of nonlinear parabolic equation \eqref{e1.1}
under certain assumptions on a-priori estimates can be summarized as follows.

\begin{theorem}\label{t1.1}
For any $T>0$, we assume that  $u\in C^{4+\alpha,\frac{4+\alpha}{2}}(\bar{\Omega}_{T})$
be a unique solution of the nonlinear parabolic equation (\ref{e1.1}) which satisfy
\begin{equation}\label{e1.4}
\|u_{t}(\cdot,t)\|_{C(\bar{\Omega})}+\|Du(\cdot,t)\|_{C(\bar{\Omega})}+\|D^{2}u(\cdot,t)\|_{C(\bar{\Omega})}\leq C_{1},
\end{equation}
\begin{equation}\label{e1.40}
\|D^{2}u(\cdot,t)\|_{C^{\alpha}(\bar{D})}\leq C_{2},\quad \forall D\subset\subset\Omega,
\end{equation}
\begin{equation}\label{e1.5}
\sum_{k=1}^{n}h_{p_{k}}(Du(\cdot,t),x)\nu_{k}\geq \frac{1}{C_{3}},
\end{equation}
where the positive constants $C_{1}$, $C_{2}$ and $C_{3}$  are independent of ~ $t\geq 1$.
Then  $u(\cdot,t)$ converges to a function $u^{\infty}(x,t)=\tilde{u}^\infty(x)+C_{\infty}\cdot t$
in $C^{1+\zeta}(\bar{\Omega})\cap C^{4+\alpha'}(\bar{D})$ as $t\rightarrow\infty$
  for any $D\subset\subset\Omega$, $\zeta<1$ and $\alpha'<\alpha$, that is
 $$\lim_{t\rightarrow+\infty}\|u(\cdot,t)-u^{\infty}(\cdot,t)\|_{C^{1+\zeta}(\bar{\Omega})}=0,\qquad
  \lim_{t\rightarrow+\infty}\|u(\cdot,t)-u^{\infty}(\cdot,t)\|_{C^{4+\alpha'}(\bar{D})}=0.$$
And $\tilde{u}^{\infty}(x)\in C^{2}(\bar{\Omega})$ is a solution of
\begin{equation}\label{e1.6}
\left\{ \begin{aligned}F(D^{2}u,Du,x)&=C_{\infty},
&  x\in \Omega, \\
h(Du,x)&=0, &x\in\partial\Omega.
\end{aligned} \right.
\end{equation}
The constant $C_{\infty}$ depends only on $\Omega$ and $F$. The solution to (\ref{e1.6}) is unique up to additions of constants.
\end{theorem}
\begin{rem}\label{r1.1}
By Evans-Krylov theorem for parabolic equations \cite{W},
we can replace (\ref{e1.40}) by  $F(A,p,x)$ being concave with respect to the variable $A$.
\end{rem}

We first recall some results concerning the convergence of solutions to translating solutions in history.
Let $\Omega$ be a smooth strictly convex domain in $\mathbb{R}^{n}$. Schn$\ddot{\text{u}}$rer \cite{OC} studied a class of curvature flow in $\mathbb{R}^{n+1}$:
\begin{equation}\label{e1.7}
\left\{ \begin{aligned}\dot{X}&=-(\ln F-\ln f)\nu, \\
\nu(M)&=\nu(M_{0}),\\
M\mid_{t=0}&=M_{0},
\end{aligned} \right.
\end{equation}
where $X$ denote the embedding vector of a smooth strictly convex hypersurface with boundary, $M=\text{graph}\, u\mid_{\Omega},$
$u: \bar{\Omega}\rightarrow \mathbb{R}$, and $\dot{X}$ denote the total time derivative. By a given smooth positive function
$f:\bar{\Omega}\rightarrow \mathbb{R}$ and a curvature function $F$, Schn$\ddot{\text{u}}$rer
 transformed  the curvature flow (\ref{e1.7}) into some sort of (\ref{e1.1})
and then obtained the estimates (\ref{e1.4}), (\ref{e1.5}). Finally he proved that the initial value problem (\ref{e1.7})
admits a convex solution $M(t)=\text{graph}\, u(t)\mid_{\Omega}$ that exists for all times $t\geq 0$ and converges smoothly
to a translating solution  $M^{\infty}=\text{graph}\, u^{\infty}\mid_{\Omega}$ of the flow (\ref{e1.7}), that is, there exists
$v^{\infty}\in\mathbb{R}$ such that $$u^{\infty}(x,t)= u^{\infty}(x,0)+v^{\infty}\cdot t.$$

A similar convergence result for graphic mean curvature flow with Neumann boundary condition in arbitrary dimension
was studied by a recent work of Ma-Wang-Wei \cite{MWW}.
They studied nonparametric surfaces over strictly convex bounded domains in $\mathbb{R}^n$
which are evolving by the mean curvature flow with Neumann boundary value
\begin{equation*}
\left\{ \begin{aligned}u_t &= (\delta_{ij}-\frac{u_i u_j}{1+|Du|^2})u_{ij} \quad \text{in} \qquad \Omega\times (0,\infty), \\
u_{\nu} & = \varphi(x) \quad \text{on} \qquad \partial \Omega\times (0,\infty),\\
 u (x,0)& = u_{0}(x),
\end{aligned} \right.
\end{equation*}
where $\varphi(x)$ and $u_0(x)$ are smooth functions satisfying $u_{0,\nu}=\varphi(x)$ on $\partial \Omega$.
They proved a convergence theorem that the solution of the above mean curvature flow with Neumann boundary value converges
to a translating solution which moves at a constant speed up to a translation $\lambda t+ w $ in arbitrary dimensions,
where $(\lambda, w)$ are suitable solution to
\begin{equation*}
\left\{ \begin{aligned}  (\delta_{ij}-\frac{u_i u_j}{1+|Du|^2})u_{ij} &= \lambda \quad \text{in} \qquad \Omega, \\
u_{\nu} & = \varphi(x) \quad \text{on} \qquad \partial \Omega.
\end{aligned} \right.
\end{equation*}

Our work was also inspired by the reading of the  papers of  Altschuler-Wu \cite{SL}, and Kitagawa \cite{JK}
where the translating solutions were obtained.
In this paper we show Schn$\ddot{\text{u}}$rer's convergence result to more general case
and discuss the application of this convergence result
in the study of minimal Lagrangian diffeomorphism problem. For the proof of Theorem \ref{t1.1},
we  borrow the ideas from \cite{OC}, \cite{JK} and \cite{SL}.

The rest part of this paper is organized as follows. We give a simple case which is the convergence of the solution
to one-dimensional heat equation to illustrate the convergence in Theorem \ref{t1.1}.
We collect several preliminary results in section 3.  In section 4, we give the proof of Theorem \ref{t1.1}.
In section 5, we exhibit the applications of  Theorem \ref{t1.1} to the study of  a family of special Lagrangian graphs.

\section{A simple case}
In this section, we first give a simple case to illustrate the convergence result in Theorem \ref{t1.1}.
We consider the convergence of the solutions to the following heat equation on the interval $[0,1]$.
 \begin{equation}\label{heat}
\left\{ \begin{aligned}u_t &= u_{xx} \quad \text{for}\quad x \in [0,1],\quad t>0, \\
u_x|_{x=0} & = a ,  \quad  u_x|_{x=1} = b  ,\\
 u (x,0)& = u_{0}(x),
\end{aligned} \right.
\end{equation}
where $a$ and $b$ are two constants and $u_0(x)$ is a smooth function on the interval $[0,1]$. In physical meaning,
the model characterizes the heat conduction problem of a fine iron rod in the interval [0,1]
with heat exchanges on the two interval endpoints $0$ and $1$.
By principle of superposition in classical linear PDE theory,
we know that the solution of equation \eqref{heat} can be decomposed into three parts
\begin{equation*}
u(x,t)=V(x)+Z(t)+w(x,t),
\end{equation*}
 where $V(x)$ , $Z(t)$ and  $w(x,t)$ are solutions of the following equations respectively:
 \begin{equation}\label{heat1}
\left\{ \begin{aligned}V_{xx} &= b-a \quad \text{for}\quad x\in [0,1], \\
V_x|_{x=0} & = a ,  \quad  V_x|_{x=1} = b  ,
\end{aligned} \right.
\end{equation}

 \begin{equation}\label{heat2}
\left\{ \begin{aligned}Z_{t} &= b-a \quad \text{for}\quad t>0, \\
& Z|_{t=0} = 0,
\end{aligned} \right.
\end{equation}
and
 \begin{equation}\label{heat3}
\left\{ \begin{aligned}w_t &= w_{xx} \quad \text{for}\quad x\in [0,1],\quad t>0, \\
w_x|_{x=0} & = 0 ,  \quad  w_x|_{x=1} = 0  ,\\
 w (x,0)& = u_{0}(x)-(\frac{b-a}{2}x^2+ax).
\end{aligned} \right.
\end{equation}
By solving the equations \eqref{heat1} and \eqref{heat2}, we obtain that
\begin{equation}\label{sol1}
V(x)=\frac{b-a}{2}x^2+ax
\end{equation}
 and
 \begin{equation}\label{sol2}
 Z(t)=(b-a)t.
\end{equation}
Then it is left to solve the equation \eqref{heat3}. To that end,
we use the methods of separation of variables. We suppose the equation \eqref{heat3} has
the solutions in the following form
\begin{equation}\label{sep}
w(x,t)=X(x)T(t).
\end{equation}
Substituting \eqref{sep} into equation \eqref{heat3}, we obtain:
 \begin{equation}\label{heat4}
\frac{X''(x)}{X(x)}=\frac{T'(t)}{T(t)}=-\lambda,
\end{equation}
where $\lambda$ is a constant which will be determined later.
From equation \eqref{heat4}, we need to solve
\begin{equation}\label{heat5}
\left\{ \begin{aligned} & X''(x)+\lambda X(x) = 0 ,\quad \text{for}\quad x\in [0,1], \\
&X_x|_{x=0}= 0 ,  \quad  X_x|_{x=1} = 0  ,
\end{aligned} \right.
\end{equation}
and
\begin{equation}\label{heat6}
T'(t)+\lambda T(t) = 0 ,\quad \text{for}\quad t>0.
\end{equation}

For equation \eqref{heat5}, the eigenvalues $\lambda_n$ and corresponding eigenfunctions are $n^2\pi^2$ and $X_n(x)=C_n cos n\pi x$
for $n=0,1,2, \cdot\cdot\cdot$. For $\lambda_n$, the solutions of equation \eqref{heat6} are
 $T_n(t) = D_n e^{-n^2\pi^2 t}$ for $n=0,1,2, \cdot\cdot\cdot$. From the method of separation of variables, we know that the solution of
 equation \eqref{heat3} has the form of Fourier series:
 \begin{equation}\label{sol3}
 w(x,t)=\sum_{n=0}^{\infty} C_n e^{-n^2\pi^2 t}cos n\pi x.
 \end{equation}
 Using the orthogonality of the eigenfunctions and initial value of equation \eqref{heat3},
 we know $w(x,0)=u_0(x)-(\frac{b-a}{2}x^2+ax)$ and hence we can
 determine the coefficient constants as follows:
 \begin{equation*}
 C_n = \int_{0}^1 [u_0(x)-(\frac{b-a}{2}x^2+ax)]cos n\pi x dx,
 \end{equation*}
 for $n=0,1,2, \cdot\cdot\cdot$.
 Combining \eqref{sol1}, \eqref{sol2} and \eqref{sol3}, we know the solution to the equation \eqref{heat} is
 \begin{equation*}
 u(x,t)= (b-a)t+\frac{b-a}{2}x^2+ax + \sum_{n=0}^{\infty} C_n e^{-n^2\pi^2 t}cos n\pi x.
 \end{equation*}
 If we denote the speed constant $C_{\infty}=b-a$, Then the solution can be reformulated as
 \begin{equation*}
 u(x,t)-C_{\infty} t = \frac{b-a}{2}x^2+ax + \sum_{n=0}^{\infty} C_n e^{-n^2\pi^2 t}cos n\pi x.
 \end{equation*}
 By using the exponential decay of the exponential functions,
 we deduce that the function $u(x,t)-C_{\infty} t$
 is asymptotically approaching $\tilde{u}(x) \triangleq V(x) = \frac{b-a}{2}x^2+ax  $ up to a constant $C_0$
 and the limit function $\tilde{u}(x) $ satisfies
  \begin{equation*}
\left\{ \begin{aligned}V_{xx} &= b-a \triangleq C_{\infty} \quad \text{for}\quad x\in [0,1], \\
V_x|_{x=0} & = a ,  \quad  V_x|_{x=1} = b.
\end{aligned} \right.
\end{equation*}
Therefore we have completed the proof that the solution of the heat equation \eqref{heat} $u(x,t)$ converges to a function $u^{\infty}(x,t)=\tilde{u}^\infty(x)+C_{\infty}\cdot t$
as $t\rightarrow\infty$.

\section{Preliminary results}
In this section, we collect several preliminary results which will be used to prove
Theorem \ref{t1.1} in the next section.

\begin{lemma}\label{l2.1}
For any $T>0$, we assume that $a^{ij}\in C(\bar{\Omega}_{T})$, $b^{i}\in C(\bar{\Omega}_{T})$ and $[a^{ij}]|_{\bar{\Omega}_{T}}>0$.
If $w\in C^{2,1}(\bar{\Omega}_{T})$ satisfies
\begin{equation}\label{e2.1}
\left\{ \begin{aligned}w_{t}-a^{ij}w_{ij}-b^{i}w_{i}&=0,
&  x\in \Omega,\,\, t>0, \\
\beta^{k}w_{k}&=0,& x\in \partial\Omega,\,\, t>0,\\
\end{aligned} \right.
\end{equation}
where $\beta=(\beta^{1},\beta^{2},\cdots,\beta^{n})$ is a uniformly strictly oblique vector field,
that is, $\beta^{k}\nu_{k}\geq \frac{1}{C}$. Then
\begin{equation*}
osc_{w}(t)\triangleq \max_{x\in \bar{\Omega}}w(x,t)-\min_{x\in \bar{\Omega}}w(x,t)
\end{equation*}
is a strictly decreasing function or $w$ is a constant
\end{lemma}
\begin{proof}
Suppose $w$ is not a constant, then we can prove that $\min_{x \in \Omega}w(x,t)$ is a strictly increasing function and
$\max_{x \in \Omega}w(x,t)$ is a strictly decreasing function with respect to $t$.
In fact, since $w(x,t)$ is a solution of the parabolic equation \eqref{e2.1},
then for any $t>t_0$, by using the strong maximum principles \cite{W}, we obtain
\begin{equation*}
w(x,t)>\min_{(x,t)\in \partial \Omega\times[t_0,T)\cup \Omega\times \{t_0\}}w(x,t),\,\, \forall x\in \Omega,\,\, t_0<t<T.
\end{equation*}
By the boundary condition of equation \eqref{e2.1} and the Hopf Lemma, we know that $w(x,t)$
can not attain its minimum on the side boundary $\partial \Omega\times [t_0,T)$, then we deduce
 \begin{equation*}
w(x,t)>\min_{x\in \Omega}w(x,t_0),\,\, \forall x\in \Omega,\,\, t>t_0.
\end{equation*}
Taking the minimum of $x\in \Omega$, we deduce that
\begin{equation}\label{e2.1a}
\min_{x\in \Omega} w(x,t)>\min_{x\in \Omega}w(x,t_0),\,\, \forall t>t_0.
\end{equation}
Similarly, we have
\begin{equation}\label{e2.1b}
\max_{x\in \Omega} w(x,t)<\max_{x\in \Omega}w(x,t_0),\,\, \forall t>t_0.
\end{equation}
Combining \eqref{e2.1a} and \eqref{e2.1b}, we deduce that for $t>t_0$,
\begin{equation*}
osc_{w}(t)\triangleq \max_{x\in \bar{\Omega}}w(x,t)-\min_{x\in \bar{\Omega}}w(x,t)
<\max_{x\in \bar{\Omega}}w(x,t_0)-\min_{x\in \bar{\Omega}}w(x,t_0)=osc_{w}(t_0),
\end{equation*}
which shows that $osc_{w}(t)$ is a strictly decreasing function.
\end{proof}

\begin{lemma}\label{l2.2}
For any sequence $\{t_{n}\} \,\, (t_{n}\rightarrow +\infty)$ and for $x_{0}\in \bar{\Omega}$ and $t_{0}>0$  fixed, there exists a subsequence of $\{t_{n}\}$ (again denoted by itself), such that
\begin{equation*}
u(x,t+t_{n})-u(x_{0},t_{n}) \quad \text{and}\quad u(x,t+t_{0}+t_{n})-u(x_{0},t_{0}+t_{n})
\end{equation*}
 converge to two functions which we denote by $u^{\infty}$ and $u^{t_{0},\infty}$ respectively  in the following sense,
\begin{equation}\label{e2.2}
\begin{aligned}
 &\lim_{n\rightarrow+\infty}\|u(x,t+t_{n})-u(x_{0},t_{n})-u^{\infty}(x,t)\|_{C^{1+\zeta}(\bar{\Omega})}=0,\\
 &\lim_{n\rightarrow+\infty}\|u(x,t+t_{0}+t_{n})-u(x_{0},t_{0}+t_{n})-u^{t_{0},\infty}(x,t)\|_{C^{1+\zeta}(\bar{\Omega})}=0,\\
 &\lim_{n\rightarrow+\infty}\|u(x,t+t_{n})-u(x_{0},t_{n})-u^{\infty}(x,t)\|_{C^{2}(\bar{D})}=0,\\
 &\lim_{n\rightarrow+\infty}\|u(x,t+t_{0}+t_{n})-u(x_{0},t_{0}+t_{n})-u^{t_{0},\infty}(x,t)\|_{C^{2}(\bar{D})}=0,\\
 \end{aligned}
\end{equation}
where $D\subset\subset\Omega$, $0<\zeta<1$.
\end{lemma}
\begin{proof}
By making use of the intermediate value theorem, we obtain
\begin{equation*}
\begin{aligned}
&|u(x,t+t_{n})-u(x_{0},t_{n})|\\
=&|u(x,t+t_{n})-u(x_{0},t+t_{n})+u(x_{0},t+t_{n})-u(x_{0},t_{n})|\\
\leq &|u(x,t+t_{n})-u(x_{0},t+t_{n})|+|u(x_{0},t+t_{n})-u(x_{0},t_{n})| \\
\leq &diam(\Omega)\|Du\|_{C(\bar{\Omega})}+t\|u_{t}(\cdot,t)\|_{C(\bar{\Omega})},
 \end{aligned}
\end{equation*}
\begin{equation*}
\begin{aligned}
&|u(x,t+t_{0}+t_{n})-u(x_{0},t_{0}+t_{n})|\\
=&|u(x,t+t_{0}+t_{n})-u(x_{0},t+t_{0}+t_{n})+u(x_{0},t+t_{0}+t_{n})-u(x_{0},t_{0}+t_{n})|\\
\leq &|u(x,t+t_{0}+t_{n})-u(x_{0},t+t_{0}+t_{n})|+|u(x_{0},t+t_{0}+t_{n})-u(x_{0},t_{0}+t_{n})| \\
\leq &diam(\Omega)\|Du\|_{C(\bar{\Omega})}+t\|u_{t}(\cdot,t)\|_{C(\bar{\Omega})}.
 \end{aligned}
\end{equation*}
Combining the above two inequality with (\ref{e1.4}), we have
\begin{equation}\label{e2.3}
\begin{aligned}
\|u(\cdot,t+t_{n})-u(x_{0},t_{n})\|_{C^{2}(\bar{\Omega})}\leq C,\\
\|u(\cdot,t+t_{0}+t_{n})-u(x_{0},t_{0}+t_{n})\|_{C^{2}(\bar{\Omega})}\leq C,
 \end{aligned}
\end{equation}
where the constant $C$ is independent of $t_{n}$.

By $Arzel\grave{a}-Ascoli$ theorem, there exist functions $u^{\infty}$ and $u^{t_{0},\infty}$, such that
\begin{equation}\label{e2.4}
\begin{aligned}
 &\lim_{n\rightarrow+\infty}\|u(x,t+t_{n})-u(x_{0},t_{n})-u^{\infty}(x,t)\|_{C^{1+\zeta}(\bar{\Omega})}=0,\\
 &\lim_{n\rightarrow+\infty}\|u(x,t+t_{0}+t_{n})-u(x_{0},t_{0}+t_{n})-u^{t_{0},\infty}(x,t)\|_{C^{1+\zeta}(\bar{\Omega})}=0,\\
 \end{aligned}
\end{equation}
where  $0<\zeta<1$.

By intermediate Schauder estimates for parabolic
equations (cf. Lemma 14.6  and Proposition 4.25 in \cite{GM}) and (\ref{e1.4}), for any $D\subset\subset \Omega$, we have
\begin{equation}\label{e2.40}
\begin{aligned}
\sup_{t\geq 1}\|D^{3}u(\cdot,t)\|_{C(\bar{D})}\leq C,
\end{aligned}
\end{equation}
where $C$ is the constant depending only on  the known data and dist$(\partial \Omega, \partial D)$. Using (\ref{e2.3}),
we obtain
\begin{equation*}
\begin{aligned}
\|u(\cdot,t+t_{n})-u(x_{0},t_{n})\|_{C^{3}(\bar{D})}\leq C,\\
\|u(\cdot,t+t_{0}+t_{n})-u(x_{0},t_{0}+t_{n})\|_{C^{3}(\bar{D})}\leq C,
 \end{aligned}
\end{equation*}
where the constant $C$ is independent of $t_{n}$. By $Arzel\grave{a}-Ascoli$ theorem again, we get
\begin{equation}\label{e2.5}
\begin{aligned}
 &\lim_{n\rightarrow+\infty}\|u(x,t+t_{n})-u(x_{0},t_{n})-u^{\infty}(x,t)\|_{C^{2}(\bar{D})}=0,\\
 &\lim_{n\rightarrow+\infty}\|u(x,t+t_{0}+t_{n})-u(x_{0},t_{0}+t_{n})-u^{t_{0},\infty}(x,t)\|_{C^{2}(\bar{D})}=0.\\
 \end{aligned}
\end{equation}
Putting (\ref{e2.4}) and (\ref{e2.5}) together, we obtain the desired results.
\end{proof}
Define
\begin{equation*}
osc_{f}(t)\triangleq \max_{x\in \bar{\Omega}}f(x,t)-\min_{x\in \bar{\Omega}}f(x,t),
\end{equation*}
for any $f\in C(\bar{\Omega}\times [0,+\infty))$.  We fix the positive constant $t_0$ and write
$$w(x,t)=u(x,t)-u(x,t+t_0).$$
Then we have
\begin{equation}\label{e2.6}
\begin{aligned}
\dot{w}&=F(D^2u(x,t), Du(x,t),x)-F(D^2u(x,t+t_{0}), Du(x,t+t_{0}),x)\\
&=\int_0^1\frac{d}{ds}F(sD^2u(x,t)+(1-s)D^2u(x,t+t_{0}),sDu(x,t)+(1-s)Du(x,t+t_{0}),x)ds\\
&=\big(\int_0^1 \nabla_{r_{ij}}F(sD^2u(x,t)+(1-s)D^2u(x,t+t_{0}),sDu(x,t)+(1-s)Du(x,t+t_{0}),x)ds\big)w_{ij}\\
&+\big(\int_0^1 \nabla_{p_{i}}F(sD^2u(x,t)+(1-s)D^2u(x,t+t_{0}),sDu(x,t)+(1-s)Du(x,t+t_{0}),x)ds\big)w_{i}\\
&=a^{ij}w_{ij}+b^{i}w_{i},\qquad x\in\Omega, \quad t>0,
\end{aligned}
\end{equation}
where we denote
\begin{equation*}
\begin{aligned}
a^{ij}&=\int_0^1 \nabla_{r_{ij}}F(sD^2u(x,t)+(1-s)D^2u(x,t+t_{0}),sDu(x,t)+(1-s)Du(x,t+t_{0}),x)ds,\\
b^{i}&= \int_0^1 \nabla_{p_{i}}F(sD^2u(x,t)+(1-s)D^2u(x,t+t_{0}),sDu(x,t)+(1-s)Du(x,t+t_{0}),x)ds.
\end{aligned}
\end{equation*}
Calculating in a similar way, we know $w$ satisfies the following boundary condition,
\begin{equation}\label{e2.7}
\begin{aligned}
0&=h(Du(x,t),x)-h(Du(x,t+t_{0}),x)\\
&=\int_0^1\frac{d}{ds}h(sDu(x,t)+(1-s)Du(x,t+t_{0}),x)ds\\
&=\beta^{i}w_{i},\qquad x\in\partial\Omega, \quad t>0,
\end{aligned}
\end{equation}
where
\begin{equation*}
\begin{aligned}
\beta^{i}&= \int_0^1 \nabla_{p_{i}}h(sDu(x,t)+(1-s)Du(x,t+t_{0}),x)ds.
\end{aligned}
\end{equation*}
By (\ref{e1.4}) and (\ref{e1.5}), there exists $N\in Z^{+}$ depending only on the property of $h(\cdot,x)$ in $\partial\Omega$, such that if $0<t_{0}\leq\frac{1}{N}$, then  $a^{ij}, b^{i}, $ and $\beta^{i}$ satisfy the conditions in Lemma \ref{l2.1}.

\begin{lemma}\label{l2.3}
$u^{\infty}$ and $u^{t_{0},\infty}$ satisfy the following evolution equation with second boundary condition:
\begin{equation}\label{e2.8}
\left\{ \begin{aligned}u_{t}-F(D^{2}u,Du,x)&=0,
&  x\in \Omega,\,\, t>0, \\
h(Du,x)&=0,& x\in \partial\Omega,\,\, t>0,\\
\end{aligned} \right.
\end{equation}
and
\begin{equation*}
osc_{[u^{\infty}-u^{t_{0},\infty}]}(t)
\end{equation*}
is a strictly decreasing function or $u^{\infty}-u^{t_{0},\infty}$ is constant. In addition, $u^{\infty}$ and $u^{t_{0},\infty}$  satisfy the estimate (\ref{e1.4}),(\ref{e1.40}) and (\ref{e1.5}).
\end{lemma}
\begin{proof}
We first point out that $u^{\infty}$ satisfy the evolution equation with second boundary condition.
In fact, if we denote $w_n(x,t) = u(x,t+t_{n})-u(x_{0},t_{n})$, then it satisfies equation \eqref{e2.8}.
For any $x \in \Omega$, we can choose an open subset $D$ with compact support in $\Omega$ such that $x \in D\subset \bar{D}\subset \Omega$.
By \eqref{e2.2} in Lemma \ref{l2.2}, we know $w_n$ and $F(D^2w_n,Dw_n,x)$ uniformly converge to $u^{\infty}$ and $F(D^2u^{\infty},u^{\infty},x)$ in $D$
respectively as $n\rightarrow \infty$.

 By letting $n\rightarrow \infty$ in $(w_n)_t = F(D^2w_n,Dw_n,x)$, we obtain
 \begin{equation*}
 \begin{split}
& u^{\infty}(x,t)-u^{\infty}(x,t_0)=\lim_{n\rightarrow \infty} ( w_n(x,t)-w_n(x,t_o))\\
= & \lim_{n\rightarrow \infty} \int_{t_o}^t  (w_n)_s ds =\lim_{n\rightarrow \infty} \int_{t_o}^t  F(D^2 w_n, D w_n,x) ds\\
= &  \int_{t_o}^t \lim_{n\rightarrow \infty} F(D^2 w_n, D w_n,x) ds = \int_{t_o}^t  F(D^2 u^{\infty}, D u^{\infty},x) ds .
 \end{split}
 \end{equation*}
Taking derivative with respect to the variable, we deduce that $u^{\infty}_t= F(D^2 u^{\infty}, D u^{\infty},x)$.
By uniform convergence of $w_n(x,t)$ to $u^{\infty}(x,t)$ in $C^{1,\zeta}(\bar{\Omega})$ as $n\rightarrow \infty$
 in \eqref{e2.2}, we easily derive that $h(D u^{\infty},x)=0$ on $\partial \Omega$. Hence $u^{\infty}$ satisfies equation \eqref{e2.8}.
Similarly, $u^{t_0,\infty}$ also satisfies equation \eqref{e2.8}.

If we set  $$w(x,t)=u^{\infty}(x,t)-u^{t_{0},\infty}(x,t),$$
then by \eqref{e2.6} and \eqref{e2.7}, we know $w(x,t)$ satisfies
\begin{equation*}
\left\{ \begin{aligned}w_{t}-a^{ij}w_{ij}-b^{i}w_{i}&=0,
&  x\in \Omega,\,\, t>0, \\
\beta^{k}w_{k}&=0,& x\in \partial\Omega,\,\, t>0.\\
\end{aligned} \right.
\end{equation*}
By Lemma \ref{l2.1}, we deduce that $osc_{[u^{\infty}-u^{t_{0},\infty}]}(t)$ is a strictly decreasing function
or $u^{\infty}-u^{t_{0},\infty}$ is constant.
Furthermore, since $u\in C^{3+\alpha,\frac{3+\alpha}{2}}(\bar{\Omega}_{T})$,
by using \eqref{e2.40}, \eqref{e2.5} and  $Arzel\grave{a}-Ascoli$ theorem,
we know that  $u^{\infty}$ and $u^{t_{0},\infty}$  satisfy the estimate (\ref{e1.4}),(\ref{e1.40}) and (\ref{e1.5}).
\end{proof}

\begin{lemma}\label{l2.4}
If we take any  positive constant $0<t_0\leq \frac{1}{N}$ for some large $N>0$  depending only on the property of $h(\cdot,x)$ in $\partial\Omega$ and write $$w(x,t)=u(x,t)-u(x,t+t_0),$$
then we have
$\lim_{t\rightarrow+\infty}osc_{w}(t)=0$.
\end{lemma}
\begin{proof}
By Lemma \ref{l2.1}, $\lim_{t\rightarrow+\infty}osc_{w}(t)$ exists.

If $\lim_{t\rightarrow+\infty}osc_{w}(t)=\epsilon>0.$
Then $w$ is not a constant function. For any sequence $\{t_{n}\} \,\, (t_{n}\rightarrow +\infty)$
and for $x_{0}\in \bar{\Omega}$ fixed, let
\begin{equation*}
\begin{aligned}
\varphi_{n}(x,t)&=u(x,t+t_{n})-u(x_{0},t_{n}),\\
\varphi_{t_{0},n}(x,t)&=u(x,t+t_{0}+t_{n})-u(x_{0},t_{0}+t_{n}).
\end{aligned}
\end{equation*}
It is easy to check that,
\begin{equation}\label{e2.9}
osc_{[\varphi_{n}-\varphi_{t_{0},n}]}(t)=osc_{w}(t+t_{n}).
\end{equation}
There exists a subsequence of $\{t_{n}\}$ (again denoted by itself), such that
\begin{equation*}
\lim_{n\rightarrow+\infty}osc_{[\varphi_{n}-\varphi_{t_{0},n}]}(t)=\max_{\bar{\Omega}}(u^{\infty}-u^{t_{0},\infty})-\min_{\bar{\Omega}}(u^{\infty}-u^{t_{0},\infty}),
\end{equation*}
where we use Lemma \ref{l2.3}. On the other hand, by the assumption of the beginning ,
\begin{equation*}
\lim_{n\rightarrow+\infty}osc_{w}(t+t_{n})=\epsilon.
\end{equation*}
Then by (\ref{e2.9}), we obtain
\begin{equation*}
\begin{aligned}
osc_{[u^{\infty}-u^{t_{0},\infty}]}(t)&=\max_{\bar{\Omega}}(u^{\infty}-u^{t_{0},\infty})(t)-\min_{\bar{\Omega}}(u^{\infty}-u^{t_{0},\infty})(t)\\
&=\epsilon.
\end{aligned}
\end{equation*}
Since $ u^{\infty}$ and $u^{t_{0},\infty}$ satisfy equation \eqref{e2.8}, then we know $u^* = u^{\infty}-u^{t_{0},\infty}$
satisfies the uniformly  parabolic equation
\begin{equation*}
\left\{ \begin{aligned}u^*_{t}-a^{ij}u^*_{ij}-b^{i}u^*_{i}&=0,
&  \quad \text{in} \quad \Omega\times (-\infty,+\infty), \\
\beta^{k} u^*_{k}&=0,& \quad \text{on} \quad\partial\Omega\times (-\infty,+\infty),\\
\end{aligned} \right.
\end{equation*}
where
\begin{equation*}
\begin{aligned}
a^{ij}&=\int_0^1 \nabla_{r_{ij}}F(sD^2u^{\infty}(x,t)+(1-s)D^2u^{t_0,\infty}(x,t),sDu^{\infty}(x,t)+(1-s)Du^{t_0,\infty}(x,t),x)ds,\\
b^{i}&= \int_0^1 \nabla_{p_{i}}F(sD^2u^{\infty}(x,t)+(1-s)D^2u^{t_0,\infty}(x,t),sDu^{\infty}(x,t)+(1-s)Du^{t_0,\infty}(x,t),x)ds.
\end{aligned}
\end{equation*}

By the strong maximum principle and Hopf's Lemma,
we know $u^*= u^{\infty}-u^{t_{0},\infty}$ is a constant and hence $osc_{[u^{\infty}-u^{t_{0},\infty}]} = 0$.
This makes a contradiction to
$osc_{[u^{\infty}-u^{t_{0},\infty}]}(t) = \epsilon$ and we get the desired results.
\end{proof}
As in the proof of Lemma \ref{l2.2}, we have the following result:
\begin{lemma}\label{l2.5}
For any sequence $\{t_{n}\} \,\, (t_{n}\rightarrow +\infty)$ and for $x_{0}\in \bar{\Omega}$ fixed,
there exists a subsequence of $\{t_{n}\}$ (again denoted by itself), such that
\begin{equation*}
u(x,t+t_{n})-u(x_{0},t_{n})
\end{equation*}
 converge to   $u^{0}(x,t)$   in the following sense,
\begin{equation}\label{e2.10}
\begin{aligned}
 &\lim_{n\rightarrow+\infty}\|u(x,t+t_{n})-u(x_{0},t_{n})-u^{0}(x,t)\|_{C^{1+\zeta}(\bar{\Omega})}=0,\\
 &\lim_{n\rightarrow+\infty}\|u(x,t+t_{n})-u(x_{0},t_{n})-u^{0}(x,t)\|_{C^{2}(\bar{D})}=0,\\
 \end{aligned}
\end{equation}
where $D\subset\subset\Omega$, $0<\zeta<1$. In addition,   $u^{0}(x,t)$ satisfies (\ref{e2.8}), (\ref{e1.4}),(\ref{e1.40}) and (\ref{e1.5}).
\end{lemma}
\begin{lemma}\label{l2.6}
Let $t_{0}$ be some positive constant as in Lemma \ref{l2.4}.
There exists some constant $v^{\infty}$ such that
\begin{equation}\label{e2.11}
u^{0}(x,t+t_{0})=u^{0}(x,t)+v^{\infty}\cdot t_{0}\quad \text{for} \quad \text{any}\quad (x,t)\in \bar{\Omega}\times \{t|t\geq 0\}.
\end{equation}
\end{lemma}
\begin{proof}
By  making use of Lemma \ref{l2.5}, we have
\begin{equation}\label{e2.12}
\begin{aligned}
 &\lim_{n\rightarrow+\infty}u(x,t+t_{n})-u(x_{0},t_{n})=u^{0}(x,t),\\
 &\lim_{n\rightarrow+\infty}u(x,t+t_{0}+t_{n})-u(x_{0},t_{n})=u^{0}(x,t+t_{0})\\
 \end{aligned}
\end{equation}
for any  $(x,t)\in \bar{\Omega}\times \{t|t\geq 0\}$. On the other hand,
using Lemma \ref{l2.4} we conclude that
as $t\rightarrow+\infty$, $u(x,t)-u(x,t+t_{0})$ converges to some constant uniformly in $\bar{\Omega}$. We denote the process by
\begin{equation*}
\lim_{t\rightarrow+\infty}(u(x,t)-u(x,t+t_{0}))=-v^{\infty}\cdot t_{0}.
\end{equation*}
Replacing $t$ by $t+t_{n}$ in the above, we obtain
\begin{equation}\label{e2.13}
\lim_{n\rightarrow+\infty}(u(x,t+t_{n})-u(x,t+t_{n}+t_{0}))=-v^{\infty}\cdot t_{0}.
\end{equation}
Combining (\ref{e2.12}) with (\ref{e2.13}),  we get
\begin{equation*}
\begin{aligned}
u^{0}(x,t+t_{0})-u^{0}(x,t)&=\lim_{n\rightarrow+\infty}(u(x,t+t_{0}+t_{n})-u(x_{0},t_{n}))\\
&-\lim_{n\rightarrow+\infty}(u(x,t+t_{n})-u(x_{0},t_{n}))\\
&=\lim_{n\rightarrow+\infty}(u(x,t+t_{n}+t_{0})-u(x,t+t_{n}))\\
&=v^{\infty}\cdot t_{0}.
\end{aligned}
\end{equation*}
\end{proof}
Let
\begin{equation*}
P=\{p_{i}|p_{0}<p_{1}<\cdots,i=0,1,\cdots\}
\end{equation*}
be the set of all positive prime numbers.

By Lemma \ref{l2.5}, $u^{0}(x,t)$ satisfies (\ref{e2.8}), (\ref{e1.4}), (\ref{e1.40}) and (\ref{e1.5}).
If we replace $(u, t_{1})$ by $(u^{0}, t_{0})$  and repeat the argument as in the proof of Lemma \ref{l2.6},
 we have
\begin{lemma}\label{l2.7}
There exist $N\in Z^{+}$  depending only on the property of $h(\cdot,x)$ in $\partial\Omega$
and the function $u^{1}(x,t)$ satisfying (\ref{e2.8}), (\ref{e1.4}), (\ref{e1.40}) and (\ref{e1.5}), such that we obtain
\begin{equation}\label{e2.14}
u^{1}(x,t+t_{1})=u^{1}(x,t)+\tilde{v}^{\infty}\cdot t_{1}\quad \text{for} \quad \text{any}\quad (x,t)\in \bar{\Omega}\times \{t|t\geq 0\},
\end{equation}
\begin{equation}\label{e2.15}
u^{1}(x,t+t_{0})=u^{1}(x,t)+v^{\infty}\cdot t_{0}\quad \text{for} \quad \text{any}\quad (x,t)\in \bar{\Omega}\times \{t|t\geq 0\},
\end{equation}
where $t_{0}=\frac{1}{Np_{0}}, t_{1}=\frac{1}{Np_{1}}, \tilde{v}^{\infty}=v^{\infty}$.
\end{lemma}
\begin{proof}
As Lemma \ref{l2.5},  there exist sequence $\{t_{n}\} \,\, (t_{n}\rightarrow +\infty)$ for  $x_{0}\in \bar{\Omega}$ to be  fixed and $u^{1}(x,t)$, such that
\begin{equation}\label{e2.16}
\begin{aligned}
 &\lim_{n\rightarrow+\infty}\|u^{0}(x,t+t_{n})-u^{0}(x_{0},t_{n})-u^{1}(x,t)\|_{C^{1+\zeta}(\bar{\Omega})}=0,\\
 &\lim_{n\rightarrow+\infty}\|u^{0}(x,t+t_{n})-u^{0}(x_{0},t_{n})-u^{1}(x,t)\|_{C^{2}(\bar{D})}=0,\\
 \end{aligned}
\end{equation}
where $\forall t\geq 0, D\subset\subset\Omega$, $0<\zeta<1$,  In addition,   $u^{1}(x,t)$ satisfies (\ref{e2.8}), (\ref{e1.4}) and (\ref{e1.5}).
By (\ref{e2.16}), we also have
\begin{equation}\label{e2.17}
\begin{aligned}
 &\lim_{n\rightarrow+\infty}\|u^{0}(x,t+t_{0}+t_{n})-u^{0}(x_{0},t_{n})-u^{1}(x,t+t_{0})\|_{C^{1+\zeta}(\bar{\Omega})}=0,\\
 &\lim_{n\rightarrow+\infty}\|u^{0}(x,t+t_{0}+t_{n})-u^{0}(x_{0},t_{n})-u^{1}(x,t+t_{0})\|_{C^{2}(\bar{D})}=0.\\
 \end{aligned}
\end{equation}
Combining (\ref{e2.16}) with (\ref{e2.17}), we obtain
\begin{equation}\label{e2.18}
\begin{aligned}
u^{1}(x,t+t_{0})-u^{1}(x,t)&=\lim_{n\rightarrow+\infty}(u^{0}(x,t+t_{0}+t_{n})-u^{0}(x_{0},t_{n}))\\
&-\lim_{n\rightarrow+\infty}(u^{0}(x,t+t_{n})-u^{0}(x_{0},t_{n}))\\
&=\lim_{n\rightarrow+\infty}(u^{0}(x,t+t_{0}+t_{n})-u^{0}(x,t+t_{n}))\\
&=v^{\infty}\cdot t_{0},\quad \text{for} \quad \text{any}\quad (x,t)\in \bar{\Omega}\times \{t|t\geq 0\},
\end{aligned}
\end{equation}
where we use Lemma \ref{l2.6}.

On the other hand, for any  positive constant $0<t_{1}\leq \frac{1}{N}$ for some large $N>0$  depending only
on the property of $h(\cdot,x)$ in $\partial\Omega$, as in the proof of Lemma \ref{l2.6}, there exists some constant $\tilde{v}^{\infty}$ such that
\begin{equation*}
\lim_{t\rightarrow+\infty}(u^{0}(x,t)-u^{0}(x,t+t_{1}))=-\tilde{v}^{\infty}\cdot t_{1}, \quad \text{uniformly}\,\,\, \text{in}\,\,\, \bar{\Omega}.
\end{equation*}
Replacing $t$ by $t+t_{n}$ in the above, we obtain
\begin{equation*}
\lim_{n\rightarrow+\infty}(u^{0}(x,t+t_{n})-u^{0}(x,t+t_{n}+t_{1}))=-\tilde{v}^{\infty}\cdot t_{1}.
\end{equation*}
Then we obtain
\begin{equation}\label{e2.19}
\begin{aligned}
\tilde{v}^{\infty}\cdot t_{1}&=\lim_{n\rightarrow+\infty}(u^{0}(x,t+t_{n}+t_{1})-u^{0}(x,t+t_{n}))\\
&=\lim_{n\rightarrow+\infty}(u^{0}(x,t+t_{n}+t_{1})-u^{0}(x_{0},t_{n})\\
&+u^{0}(x_{0},t_{n})-u^{0}(x,t+t_{n}))\\
&=\lim_{n\rightarrow+\infty}(u^{0}(x,t+t_{n}+t_{1})-u^{0}(x_{0},t_{n}))\\
&+\lim_{n\rightarrow+\infty}(u^{0}(x_{0},t_{n})-u^{0}(x,t+t_{n}))\\
&=u^{1}(x,t+t_{1})-u^{1}(x,t)\quad \text{for} \quad \text{any}\quad (x,t)\in \bar{\Omega}\times \{t|t\geq 0\},
\end{aligned}
\end{equation}
where we use (\ref{e2.16}).

By (\ref{e2.18}) and (\ref{e2.19}), for any $m_{0}, m_{1}\in Z^{+}$, we get
\begin{equation}\label{e2.20}
\begin{aligned}
u^{1}(x,t+m_{0}t_{0})-u^{1}(x,t)&=v^{\infty}\cdot m_{0}t_{0},\\
u^{1}(x,t+m_{1}t_{1})-u^{1}(x,t)&=\tilde{v}^{\infty}\cdot m_{1}t_{1},
\end{aligned}
\end{equation}
for any $(x,t)\in \bar{\Omega}\times \{t|t\geq 0\}$. By the conditions of the choice according to $t_{0}$ and $t_{1}$,
there exist $N\in Z^{+}$  depending only on the property of $h(\cdot,x)$ in $\partial\Omega$ such that we can take
$t_{0}=\frac{1}{Np_{0}}$ and $t_{1}=\frac{1}{Np_{1}}$ respectively. Letting $m_{0}=Np_{0}, m_{1}=Np_{1}, t=0$, then by (\ref{e2.20}) we obtain
\begin{equation*}
\begin{aligned}
u^{1}(x,1)-u^{1}(x,0)&=v^{\infty},\\
u^{1}(x,1)-u^{1}(x,0)&=\tilde{v}^{\infty}.
\end{aligned}
\end{equation*}
So $\tilde{v}^{\infty}=v^{\infty}.$
\end{proof}
By the proof of Lemma \ref{l2.7}, we deduce that
\begin{lemma}\label{l2.8}
There exist $N\in Z^{+}$  depending only on the property of $h(\cdot,x)$ in $\partial\Omega$
and the function $u^{i}(x,t)(i=0,1,\cdots)$ satisfying (\ref{e2.8}), (\ref{e1.4}), (\ref{e1.40})and (\ref{e1.5}), such that we obtain
\begin{equation}\label{e2.21}
\begin{aligned}
&u^{i}(x,t+\tau_{k})=u^{i}(x,t)+v^{\infty}\cdot \tau_{k}(k=0,1,\cdots,i)\\
&\text{for} \quad \text{any}\quad (x,t)\in \bar{\Omega}\times \{t|t\geq 0\},
\end{aligned}
\end{equation}
where $\tau_{k}=\frac{1}{Np_{k}}$ and $v^{\infty}$ is some constant.
\end{lemma}
\begin{lemma}\label{l2.9}
There exist $u^{\infty}$ satisfying (\ref{e2.8}), (\ref{e1.4}), (\ref{e1.40})and (\ref{e1.5}), such that we obtain
\begin{equation}\label{e2.22}
\begin{aligned}
&u^{\infty}(x,t+\tau)=u^{\infty}(x,t)+v^{\infty}\cdot \tau\\
&\text{for} \quad \text{any}\quad (x,t)\in \bar{\Omega}\times \{t|t\geq 0\},\,\,\tau\geq 0,
\end{aligned}
\end{equation}
where  $v^{\infty}$ is some constant.
\end{lemma}
\begin{proof}
By $Arzel\grave{a}-Ascoli$ theorem, we consider a diagonal sequence of \{i\} (again denoted by itself), such that
\begin{equation}\label{e2.23}
\begin{aligned}
 &\lim_{i\rightarrow+\infty}\|u^{i}(\cdot,t)-u^{\infty}(\cdot,t)\|_{C^{1+\zeta}(\bar{\Omega})}=0,\quad \forall\, t\geq 0,\\
 &\lim_{i\rightarrow+\infty}\|u^{i}(\cdot,t)-u^{\infty}(\cdot,t)\|_{C^{2}(\bar{D})}=0,\quad\forall\, t\geq 0,\\
 \end{aligned}
\end{equation}
for some function $u^{\infty}$ satisfies (\ref{e2.8}), (\ref{e1.4}), (\ref{e1.40})and (\ref{e1.5}).

By (\ref{e2.21}) in Lemma \ref{l2.8}, we get
\begin{equation*}
\begin{aligned}
&u^{\infty}(x,t+\tau_{k})=u^{\infty}(x,t)+v^{\infty}\cdot \tau_{k}~(k=0,1,\cdots,)\\
&\text{for} \quad \text{any}\quad (x,t)\in \bar{\Omega}\times \{t|t\geq 0\}.
\end{aligned}
\end{equation*}
It is easy to see that
\begin{equation*}
\begin{aligned}
&u^{\infty}(x,t+m\tau_{k})=u^{\infty}(x,t)+v^{\infty}\cdot m\tau_{k}~(k=0,1,\cdots,)\\
&\text{for} \quad \text{any}\quad (x,t)\in \bar{\Omega}\times \{t|t\geq 0\},\,\,m\in Z^{+}.
\end{aligned}
\end{equation*}
Then
\begin{equation*}
\begin{aligned}
&u^{\infty}(x,t+q)=u^{\infty}(x,t)+v^{\infty}\cdot q~(k=0,1,\cdots,)\\
&\text{for} \quad \text{any}\quad (x,t)\in \bar{\Omega}\times \{t|t\geq 0\},\,\,q\in Q^{+},
\end{aligned}
\end{equation*}
where $ Q^{+}=\{q\mid q=\frac{p_{i}}{p_{j}},\,\,p_{i}, p_{j}\in P\}.$
By the denseness of $Q^{+}$ in $\{\tau\geq 0\}$, we obtain the desired results.
\end{proof}
\begin{lemma}\label{l2.10}
\begin{equation*}
\lim_{t\rightarrow+\infty}\|u(\cdot,t)-u^{\infty}(\cdot,t)\|_{C(\bar{\Omega})}=0.
\end{equation*}
\end{lemma}
\begin{proof}
 Let
\begin{equation*}
W(x,t)=u(x,t)-u^{\infty}(x,t).
\end{equation*}
If we replace $u(x,t+t_{0})$ by $u^{\infty}(x,t)$ and repeat the arguments in  Lemma \ref{l2.4}, then we have the desired results.
\end{proof}

\section{Proof of Theorem \ref{t1.1}}
In this section, we give the proof of Theorem \ref{t1.1} by using the preliminary results which we establish in section 3.

We first give an estimates which will be used in the following proof.
 By intermediate Schauder estimates for parabolic
equations (cf. Lemma 14.6  and Proposition 4.25 in \cite{GM}), for any $D\subset\subset \Omega$, we have
\begin{equation}\label{e4.209}
\begin{aligned}
&\sup_{t\geq 1}\|D^{3}u(\cdot,t)\|_{C(\bar{D})}+\sup_{t\geq 1}\|D^{4}u(\cdot,t)\|_{C(\bar{D})}\\
&+\sup_{x_{i}\in D, t_{i}\geq 1}\frac{|D^{4}u(x_{1}, t_{1})-D^{4}u(x_{2}, t_{2})}{\max\{|x_{1}-x_{2}|^{\alpha},|t_{1}-t_{2}|^{\frac{\alpha}{2}}\}}\\
&\leq C_{4},
\end{aligned}
\end{equation}
where $ C_{4}$ is a constant depending on the known data and dist$(\partial \Omega, \partial D)$.

Now fixing some positive $t_0$ and writing $$w(x,t)=u(x,t)-u(x,t+t_0),$$ then we see that $w(x,t)$ satisfies equation \eqref{e2.1}.
By \eqref{e1.4} and \eqref{e1.40}, we know the equation \eqref{e2.1} are uniformly parabolic.
Since $h_{p_k}(x,Du(x,t))\nu_k\geq C>0$ for some $C$ uniform in $t$ and $x$.
Therefore we can choose $t_0$ small enough to ensure that $\alpha^k\nu_k\geq \frac{C}{2}>0$
and we see that $w$ satisfies a linear, uniformly oblique boundary condition.

By Lemma \ref{l2.1} to Lemma \ref{l2.9} in Section 2, we know the uniform parabolicity of equation \eqref{e2.1} ensures that
we can obtain a translating solution of the same regularity as $u$ : $u^{\infty}(x,t)=\tilde{u}^\infty(x)+C_{\infty}\cdot t$
for some constant $C_{\infty}$ which satisfies equation \eqref{e1.1}.
By Lemma \ref{l2.10},
 $u^{\infty}(x,t)$ satisfies
\begin{equation}\label{e3.1}
\lim_{t\rightarrow+\infty}\|u(\cdot,t)-u^{\infty}(\cdot,t)\|_{C(\bar{\Omega})}=0.
\end{equation}
By a result of Maz'ya-Shaposhnikova \cite{MS} on page 143 , For any $0<\alpha<1$, we have the following interpolation inequality
\begin{equation*}
|D u(x)|^{\alpha+1}\leq \frac{n(n+1)^{\alpha+1}(\alpha+1)^{\alpha+1}}{(n+\alpha)(n+\alpha+1)\alpha^{\alpha}}
(M^0(x))^{\alpha}\sup_y\frac{|D u(y)-D u(x)|}{|y-x|^{\alpha}},
\end{equation*}
where $M^0 u(x)=\sup_{\tau>0} \frac{1}{2\tau} |\int_{x-\tau}^{x+\tau} sign(x-y) u(y) dy |$ is the maximal operator.
It easily yields that for any $0<\alpha<1$,
\begin{equation*}
\|D u\|_{C(\bar{\Omega})}\leq \frac{n^{\frac{1}{\alpha+1}}(n+1)(\alpha+1)}{(n+\alpha)^{\frac{1}{\alpha+1}}(n+\alpha+1)^{\frac{1}{\alpha+1}}\alpha^{\frac{\alpha}{\alpha+1}}}
\|u\|_{C(\bar{\Omega})}^{\frac{\alpha}{\alpha+1}}\|D u \|^{\frac{1}{\alpha+1}}_{C^{\alpha}(\bar{\Omega})}.
\end{equation*}
Therefore by replacing $u(\cdot)$ with $u(\cdot,t)-u^{\infty}(\cdot,t)$ in the above inequality, we obtain
\begin{equation}\label{e3.2}
\begin{aligned}
& \|D(u(\cdot,t)-u^{\infty}(\cdot,t))\|_{C(\bar{\Omega})}\\
\leq & C(n,\alpha) \|u(\cdot,t)-u^{\infty}(\cdot,t)\|_{C(\bar{\Omega})}^{\frac{\alpha}{\alpha+1}}\|D (u(\cdot,t) - u^{\infty}(\cdot,t)) \|^{\frac{1}{\alpha+1}}_{C^{\alpha}(\bar{\Omega})}\\
\leq &  C(n,\alpha) \|u(\cdot,t)-u^{\infty}(\cdot,t)\|_{C(\bar{\Omega})}^{\frac{\alpha}{\alpha+1}}\|D^2 ( u(\cdot,t)- u^{\infty}(\cdot,t)) \|^{\frac{1}{\alpha+1}}_{C {(\bar{\Omega})}}.
\end{aligned}
\end{equation}


 From equation \eqref{e3.2} and the estimates \eqref{e1.40}, we obtain
\begin{equation}\label{e3.3}
\lim_{t\rightarrow+\infty}\|u(\cdot,t)-u^{\infty}(\cdot,t)\|_{C^{1}(\bar{\Omega})}=0.
\end{equation}
By \eqref{e1.4}, the interpolation inequalities  and (\ref{e4.209}), we get
\begin{equation*}\label{e3.4}
\begin{aligned}
&\|D^{2}(u(\cdot,t)-u^{\infty}(\cdot,t))\|^{2}_{C(\bar{D})}\\
&\leq c(D)\|D(u(\cdot,t)-u^{\infty}(\cdot,t))\|_{C(\bar{D})}(\|D^{3}(u(\cdot,t)-u^{\infty}(\cdot,t))\|_{C(\bar{D})}\\
&+\|D^{2}(u(\cdot,t)-u^{\infty}(\cdot,t))\|_{C(\bar{D})})\\
&\leq c(D)\|D(u(\cdot,t)-u^{\infty}(\cdot,t))\|_{C(\bar{D})}(2C_{4}+2C_{1}).
\end{aligned}
\end{equation*}
Using (\ref{e3.3}),  we also obtain
\begin{equation}\label{e3.5}
\lim_{t\rightarrow+\infty}\|u(\cdot,t)-u^{\infty}(\cdot,t)\|_{C^{2}(\bar{D})}=0.
\end{equation}
Repeating the above procedure and using (\ref{e4.209}),  we have
\begin{equation}\label{e3.6}
\lim_{t\rightarrow+\infty}\|u(\cdot,t)-u^{\infty}(\cdot,t)\|_{C^{4+\alpha'}(\bar{D})}=0.
\end{equation}
By \eqref{e1.4} and (\ref{e3.5}), we get
$$\lim_{t\rightarrow+\infty}\|u(\cdot,t)-u^{\infty}(\cdot,t)\|_{C^{1+\zeta}(\bar{\Omega})}=0.$$
Therefore by using the equation \eqref{e1.1} and letting $t\rightarrow\infty$, we obtain
 \begin{equation*}
C_{\infty} =\frac{\partial u^{\infty}(x,t)}{\partial t}= F(D^2u^{\infty}(x,t),Du^{\infty}(x,t),x)=F(D^2u^{\infty}(x),Du^{\infty}(x),x),
 \end{equation*}
 $$0=\lim_{t\rightarrow\infty}h(Du(x,t),x)=\lim_{t\rightarrow\infty}h(Du^{\infty}(x,t),x)=h(Du^{\infty}(x),x).$$
Then the proof of Theorem \ref{t1.1} is completed.
\qed

\section{Application of Theorem \ref{t1.1} to the convergence of special Lagrangian graphs}

In this section, we will show the applications of  Theorem \ref{t1.1} to study the existence of a family of special Lagrangian graphs.
We use the parabolic framework which consider the corresponding parabolic flows and show that
the flows converge to special Lagrangian graphs by using the convergence result.
The key steps are to verify the a-priori estimates \eqref{e1.4}, \eqref{e1.40} and the strict obliqueness estimates
of the boundary condition \eqref{e1.5}.

Elliptic methods have been successfully used to study the existence of fully-nonlinear equations with second boundary value conditions.
Elliptic Monge-Amp$\grave{e}$re equations and Hessian equations with Neumann boundary and oblique boundary conditions are solved in
\cite{P} and \cite{JU} respectively. In \cite{SM}, Brendle and Warren used the continuity method
to solve fully nonlinear elliptic equations with second boundary condition
where they obtained the existence and uniqueness of a special Lagrangian graph.

Let $\Omega$, $\tilde{\Omega}$ be two uniformly convex bounded
domains with smooth boundary in  $\mathbb{R}^{n}$  and  $a=\cot\tau$, $b=\sqrt{|\cot^{2}\tau-1|}$ for $\tau \in (0,\frac{\pi}{4}) \cup (\frac{\pi}{4},\frac{\pi}{2})$.
Here we consider the minimal Lagrangian diffeomorphism problem \cite{RS} which  is equivalent to
the following fully nonlinear elliptic equations with second boundary condition:
\begin{equation}\label{e4.1}
\left\{ \begin{aligned}F_{\tau}(D^{2}u)&=c,
&  x\in \Omega, \\
Du(\Omega)&=\tilde{\Omega},
\end{aligned} \right.
\end{equation}
where
$$F_{\tau}(A)=\left\{ \begin{aligned}
&\sum_{i}\ln\lambda_{i} , \quad
\qquad\quad\qquad\qquad\quad\quad\qquad\quad  \tau=0, \\
& \sum_{i}\ln(\frac{\lambda_{i}+a-b}{\lambda_{i}+a+b}) \qquad\qquad\quad\quad\qquad\quad 0<\tau<\frac{\pi}{4},\\
& -\sum_{i}\frac{1}{1+\lambda_{i}}, \qquad \qquad\qquad \qquad\qquad \qquad\tau=\frac{\pi}{4},\\
& \sum_{i}\arctan(\frac{\lambda_{i}+a-b}{\lambda_{i}+a+b}), \qquad \qquad\qquad \quad\frac{\pi}{4}<\tau<\frac{\pi}{2},\\
& \sum_{i}\arctan\lambda_{i}, \,\,\quad \qquad\qquad \qquad\qquad \qquad\tau=\frac{\pi}{2},
\end{aligned} \right.$$
$\lambda_i ~(1\leq i\leq n)$ are the eigenvalues of the Hessian $D^2 u$ and $Du$ are diffeomorphism from $\Omega$ to $\tilde{\Omega}$.
The existence and uniqueness of these problems have been obtained by elliptic methods in \cite{SM}, \cite{HRO} and \cite{MW}.
As in \cite{HR} and \cite{HRY}, we consider the corresponding parabolic type special Lagrangian equations and use parabolic methods to solve problem \eqref{e4.1}.
We settle the longtime existence and convergence of smooth solutions
for the following second boundary value problem to parabolic type special Lagrangian equations
\begin{equation}\label{e4.2}
\left\{ \begin{aligned}\frac{\partial u}{\partial t}&=F_{\tau}(D^{2}u),
& t>0,\quad x\in \Omega, \\
Du(\Omega)&=\tilde{\Omega}, &t>0,\qquad\qquad\\
 u&=u_{0}, & t=0,\quad x\in \Omega.
\end{aligned} \right.
\end{equation}

We will use Theorem \ref{t1.1} to prove the solutions of the
above special Lagrangian equations converging to those of problem \eqref{e4.1}. We consider different cases with respect to $\tau$.

In \cite{OK}, Schn$\ddot{\text{u}}$rer and Smoczyk studied the second boundary value problems for the Hessian and Gauss curvature flows.
They showed the convergence of the Hessian and Gauss curvature flow to the prescribed Gauss curvature equations.
We can use Theorem \ref{t1.1} to reprove part of their results.
For simplicity, we only give the proof of the special case corresponding to $\tau=0$ in equation \eqref{e4.1} and \eqref{e4.2}.
\begin{Example}
Assume that $\Omega$, $\tilde{\Omega}$ are bounded, uniformly convex domains with smooth boundary in $\mathbb{R}^{n}$, $0<\alpha<1$.
  Then for any given initial function $u_{0}\in C^{2+\alpha}(\bar{\Omega})$
  which is uniformly convex and satisfies $Du_{0}(\Omega)=\tilde{\Omega}$,  the  strictly convex solution of the following problem
  \begin{equation}\label{e4.3}
\left\{ \begin{aligned}\frac{\partial u}{\partial t}&= \ln \det D^2 u ,
& t>0,\quad x\in \Omega, \\
Du(\Omega)& = \tilde{\Omega}, &t>0,\qquad\qquad\\
 u& = u_{0}, & t=0,\quad x\in \Omega.
\end{aligned} \right.
\end{equation}
   exists for all $t\geq 0$ and $u(\cdot,t)$ converges to a function $u^{\infty}(x,t)=\tilde{u}^\infty(x)+C_{\infty}\cdot t$ in $C^{1+\zeta}(\bar{\Omega})\cap C^{4+\alpha}(\bar{D})$ as $t\rightarrow\infty$
  for any $D\subset\subset\Omega$, $\zeta<1$, that is,
 $$\lim_{t\rightarrow+\infty}\|u(\cdot,t)-u^{\infty}(\cdot,t)\|_{C^{1+\zeta}(\bar{\Omega})}=0,\qquad
  \lim_{t\rightarrow+\infty}\|u(\cdot,t)-u^{\infty}(\cdot,t)\|_{C^{4+\alpha}(\bar{D})}=0.$$
And $\tilde{u}^{\infty}(x)\in C^{\infty}(\bar{\Omega})$ is a solution of
\begin{equation}\label{e4.4}
\left\{ \begin{aligned} \ln \det D^2 u & = C_{\infty},
&  x\in \Omega, \\
Du(\Omega)&=\tilde{\Omega}.
\end{aligned} \right.
\end{equation}
The constant $C_{\infty}$ depends only on $\Omega$, $\tilde{\Omega}$ and $F$. The solution to (\ref{e4.4}) is unique up to additions of constants.
\end{Example}
\begin{proof}
From section $7$ and section $9$ in \cite{OK}, Schn$\ddot{\text{u}}$rer and Smoczyk had established estimates \eqref{e1.4}.
Since the parabolic operator is concave with with respect to the $D^2 u$ variables, by Remark \ref{r1.1}, we know \eqref{e1.40} holds.
By Lemma 8.1 in section 8 of \cite{OK}, we get the strict obliqueness estimates \eqref{e1.5}. Therefore, by Theorem \ref{t1.1},
we obtain the desired convergence results.
\end{proof}

In \cite{SM}, Brendle and Warren obtained the existence and uniqueness of  special Lagrangian graph for the special case $\tau=\frac{\pi}{2}$. In the following example,
we give an alternative proof which shows the applications of Theorem \ref{t1.1} to the existence of solutions for the equations (\ref{e4.1}) when $\tau=\frac{\pi}{2}$.

\begin{Example}
Assume that $\Omega$, $\tilde{\Omega}$ are bounded, uniformly convex domains with smooth boundary in $\mathbb{R}^{n}$, $0<\alpha<1$.
  Then for any given initial function $u_{0}\in C^{2+\alpha}(\bar{\Omega})$
  which is   uniformly convex and satisfies $Du_{0}(\Omega)=\tilde{\Omega}$,  the  strictly convex solution of the following problem
  \begin{equation}\label{e4.5}
\left\{ \begin{aligned}\frac{\partial u}{\partial t}&= \sum_{i=1}^n arctan \lambda_i,
& t>0,\quad x\in \Omega, \\
Du(\Omega)& = \tilde{\Omega}, &t>0,\qquad\qquad\\
 u& = u_{0}, & t=0,\quad x\in \Omega.
\end{aligned} \right.
\end{equation}
   exists for all $t\geq 0$ and $u(\cdot,t)$ converges to a function $u^{\infty}(x,t)=\tilde{u}^\infty(x)+C_{\infty}\cdot t$ in $C^{1+\zeta}(\bar{\Omega})\cap C^{4+\alpha}(\bar{D})$ as $t\rightarrow\infty$
  for any $D\subset\subset\Omega$, $\zeta<1$, that is,
 $$\lim_{t\rightarrow+\infty}\|u(\cdot,t)-u^{\infty}(\cdot,t)\|_{C^{1+\zeta}(\bar{\Omega})}=0,\qquad
  \lim_{t\rightarrow+\infty}\|u(\cdot,t)-u^{\infty}(\cdot,t)\|_{C^{4+\alpha}(\bar{D})}=0.$$
And $\tilde{u}^{\infty}(x)\in C^{\infty}(\bar{\Omega})$ is a solution of
\begin{equation}\label{e4.6}
\left\{ \begin{aligned}\sum_{i=1}^n arctan \lambda_i & = C_{\infty},
&  x\in \Omega, \\
Du(\Omega)&=\tilde{\Omega}.
\end{aligned} \right.
\end{equation}
The constant $C_{\infty}$ depends only on $\Omega$, $\tilde{\Omega}$ and $F$. The solution to (\ref{e1.3}) is unique up to additions of constants.
\end{Example}
\begin{proof}
By Lemma 3.1, Lemma 3.11 in \cite{HR} and boundary condition, we see that the estimates \eqref{e1.4} holds.
Since the parabolic operator is concave with respect to the $D^2 u$ variables, by Remark \ref{r1.1},
we obtain the estimate \eqref{e1.40}.
By Lemma 3.4 in \cite{HR}, we get the strict obliqueness estimates. By Theorem \ref{t1.1}, we obtain the desired results.
\end{proof}

Our next example is to show the existence of the special Lagrangian graphs for the general cases $0<\tau<\frac{\pi}{2}$.
We use Theorem \ref{e1.1} to show the convergence of the
corresponding parabolic flows to the translating solutions.
\begin{Example}
Assume that $\Omega$, $\tilde{\Omega}$ are bounded, uniformly convex domains with smooth boundary in $\mathbb{R}^{n}$, $0<\alpha<1$, $0<\tau<\frac{\pi}{2}$.
  Then for any given initial function $u_{0}\in C^{2+\alpha}(\bar{\Omega})$
  which is  uniformly convex and satisfies $Du_{0}(\Omega)=\tilde{\Omega}$,  the  strictly convex solution of the following problem
  \begin{equation}\label{e4.7}
\left\{ \begin{aligned}\frac{\partial u}{\partial t}&= F_{\tau}(D^{2}u),
& t>0,\quad x\in \Omega, \\
Du(\Omega)& = \tilde{\Omega}, &t>0,\qquad\qquad\\
 u& = u_{0}, & t=0,\quad x\in \Omega.
\end{aligned} \right.
\end{equation}
   exists for all $t\geq 0$ and $u(\cdot,t)$ converges to a function $u^{\infty}(x,t)=\tilde{u}^\infty(x)+C_{\infty}\cdot t$ in $C^{1+\zeta}(\bar{\Omega})\cap C^{4+\alpha}(\bar{D})$ as $t\rightarrow\infty$
  for any $D\subset\subset\Omega$, $\zeta<1$, that is,
 $$\lim_{t\rightarrow+\infty}\|u(\cdot,t)-u^{\infty}(\cdot,t)\|_{C^{1+\zeta}(\bar{\Omega})}=0,\qquad
  \lim_{t\rightarrow+\infty}\|u(\cdot,t)-u^{\infty}(\cdot,t)\|_{C^{4+\alpha}(\bar{D})}=0.$$
And $\tilde{u}^{\infty}(x)\in C^{\infty}(\bar{\Omega})$ is a solution of
\begin{equation}\label{e4.8}
\left\{ \begin{aligned}F_{\tau}(D^{2}u) & = C_{\infty},
&  x\in \Omega, \\
Du(\Omega)&=\tilde{\Omega}.
\end{aligned} \right.
\end{equation}
The constant $C_{\infty}$ depends only on $\Omega$, $\tilde{\Omega}$ and $F$. The solution to (\ref{e1.3}) is unique up to additions of constants.
\end{Example}
\begin{proof}
In recent papers \cite{HRY} and \cite{CHY}, the authors study second boundary value problems for a class of fully nonlinear flows with general parabolic
operators.
By Lemma 3.1, Lemma 3.11 in \cite{HRY} and boundary condition, we know the priori estimates \eqref{e1.4} holds.
Since the parabolic operator is concave with respect to the $D^2 u$ variables, by Remark \ref{r1.1}, we see that the estimate \eqref{e1.40} holds.
By Lemma 3.4 in \cite{HRY}, we know the strict obliqueness estimates hold. By Theorem \ref{t1.1}, we obtain the desired convergence results.
\end{proof}

\end{document}